%% file: Symplectic_manifold_MOR_preprint.tex
\documentclass[]{scrartcl}
\usepackage[utf8]{inputenc}
\usepackage[letterpaper,bindingoffset=0cm,inner=2.5cm,outer=2.5cm,top=2.5cm,bottom=2.5cm]{geometry}
\usepackage{helvet}
\usepackage[T1]{fontenc}
\usepackage{graphicx}
\usepackage{amsmath,amsthm,amstext,amssymb,bm}
\usepackage{mathtools}
\usepackage{xcolor,color}
\usepackage{enumerate}
\usepackage{pgf}
\usepackage[toc,page]{appendix}
\usepackage[colorlinks=true]{hyperref}
\definecolor{mylinkcolor}{RGB}{0,0,130}
\hypersetup{colorlinks,allcolors=mylinkcolor,citecolor=mylinkcolor}
\usepackage{url}
\usepackage[backend=biber,%
                      bibencoding=utf8,%
                      citestyle=numeric,%
                      giveninits=true,%
                      isbn=false,%
                      sortcites=true,%
                      natbib=true]{biblatex}
\usepackage[capitalize,nameinlink]{cleveref}
\crefformat{equation}{(#2#1#3)}
\usepackage{booktabs}
\usepackage{tabularx}
\usepackage{enumerate}
\usepackage[shortlabels]{enumitem}
\usepackage{multirow}
\usepackage{subfigure}
\usepackage[font=small,labelfont=bf]{caption}

\providecommand{\keywords}[1]{\textbf{\textit{Keywords ---}} #1}
\providecommand{\msc}[1]{\textbf{\textit{MSC 2010 ---}} #1}
\providecommand{\acknowledgements}[1]{\textbf{\textit{Acknowledgements ---}} #1}

\newtheorem{mydef}{Definition}

\newtheorem{myrem}{Remark}
\newtheorem{mythm}{Theorem}
\newtheorem{myprop}{Proposition}

\DeclarePairedDelimiter{\norm}{\lVert}{\rVert}
\DeclarePairedDelimiter{\abs}{|}{|}

\crefname{secname}{Section}{Section}
\Crefname{secname}{Sec.}{Sec.}
\crefname{section}{Section}{Sections}
\crefname{subsection}{Subsection}{Subsections}

\newcommand{\R}{\mathbb{R}}

\newcommand{\N}{\mathbb{N}}
\newcommand{\J}{\mathbb{J}}
\newcommand{\cA}{\mathcal{A}}
\newcommand{\cAE}[2]{\cA_{#1}^{#2}}
\newcommand{\bcAE}[2]{\mathbf{\cA_\textbf{#1}^{#2}}}
\newcommand{\cAEs}[1]{\cAE{s,0}{#1}}
\newcommand{\bcAEs}[1]{\bcAE{s,0}{#1}}
\newcommand{\dVPOD}[1]{{V}_\textnormal{POD}^{#1}}
\newcommand{\dVCL}[1]{{V}_\textnormal{CL}^{#1}}
\newcommand{\cI}{\mathcal{I}}
\newcommand{\cM}{\mathcal{M}}
\newcommand{\cP}{\mathcal{P}}

\newcommand{\cH}{\mathcal{H}}
\newcommand{\cN}{\mathcal{N}}

\newcommand{\bfa}{\bm{a}}

\newcommand{\bfx}{\bm{x}}
\newcommand{\tcM}{\widetilde{\mathcal{M}}}
\newcommand{\tbfx}{\widetilde{\bfx}}
\newcommand{\tbfw}{\widetilde{\bfw}}
\newcommand{\tbfr}{\widetilde{\bfr}}
\newcommand{\bfr}{\bm{r}}
\newcommand{\bfv}{\bm{v}}
\newcommand{\bfw}{\bm{w}}
\newcommand{\bfzero}{\bm{0}}
\newcommand{\bfmu}{\ensuremath{\bm{\mu}}}
\newcommand{\Jac}[1]{\ensuremath{\bm{D}_{\bfx_{r}}d(#1)}}
\newcommand{\Jacx}[1]{\ensuremath{\bm{D}_{\bfx}d(#1)}}

\newcommand{\JacC}{\bm{D}}

\newcommand{\lb}{\left(}
\newcommand{\rb}{\right)}
\newcommand{\lsb}{\left[}
\newcommand{\rsb}{\right]}

\newcommand{\textd}{\text{d}}
\newcommand{\ddvar}[1]{\frac{\textd}{{\textd}{#1}}}
\newcommand{\ddt}{\ddvar{t}}

\newcommand{\rT}[1]{\ensuremath{#1^{\textnormal{\textsf{T}}}}}
\newcommand{\rTb}[1]{\ensuremath{\rT{\lb#1\rb}}}
\newcommand{\rTsb}[1]{\ensuremath{\rT{\lsb#1\rsb}}}

\newcommand{\sI}[1]{\ensuremath{\lb#1\rb^{+}}}

\newcommand{\eproj}{e_\text{proj}}
\newcommand{\ered}{e_\text{red}}
\newcommand{\esymp}{e_\text{symp}}

\newcommand{\bfxref}{\bfx_\text{ref}}

\newcommand{\aeparam}{\bm{\theta}}
\newcommand{\saeparam}{\aeparam^\ast}
\newcommand{\naeparam}{n_{\aeparam}}
\newcommand{\dataset}{\mathcal{X}}
\newcommand{\loss}{\mathcal{L}}
\newcommand{\lossData}{\loss_\textnormal{data}}
\newcommand{\lossSympl}{\loss_\textnormal{sympl}}

\newcommand{\nconv}{n_\textnormal{conv}}
\newcommand{\nfull}{n_\textnormal{full}}

\newcommand{\convblock}{\texttt{convblock}}
\newcommand{\convTblock}{\texttt{convTblock}}

\newcommand{\fullblock}{\texttt{fullblock}}
\newcommand{\flatten}{\texttt{flat}}
\newcommand{\splitting}{\texttt{split}}
\newcommand{\scale}{\texttt{scale}}
\newcommand{\stride}{s}
\newcommand{\bfl}{\bm{l}}
\newcommand{\bfc}{\bm{c}}
\newcommand{\bfs}{\bm{s}}
\newcommand{\bfq}{\bm{q}}
\newcommand{\bfp}{\bm{p}}

\newcommand{\csconv}{\bfc}

\newcommand{\lconv}{l^\textnormal{conv}}
\newcommand{\lfull}{l^\textnormal{full}}
\newcommand{\lsconv}{\bfl^\textnormal{conv}}
\newcommand{\lsfull}{\bfl^\textnormal{full}}
\newcommand{\strides}{\bfs}

\newcommand{\eNum}[2]{#1\text{e}#2}
\newcommand{\datasetTrain}{\dataset_\textnormal{train}}
\newcommand{\datasetVal}{\dataset_\textnormal{val}}

\newcommand{\cPtrain}{\mathcal{P}_\textnormal{train}}

\newcommand{\muGenA}{\mu_1}
\newcommand{\muGenB}{\mu_2}
\newcommand{\muGenC}{\mu_3}
\newcommand{\reddims}{\mathcal{D}_r}

\author{Patrick Buchfink\thanks{Institute of Applied Analysis and Numerical Simulation, University of Stuttgart, Pfaffenwaldring 57, 70569 Stuttgart, Germany. (\url{patrick.buchfink,haasdonk@mathematik.uni-stuttgart.de})} \and Silke Glas\thanks{Department of Computer Science, Cornell University, Ithaca, NY 14853. (\url{smg374@cornell.edu})} \and Bernard Haasdonk\footnotemark[1]}
\title{Symplectic Model Reduction of Hamiltonian Systems on Nonlinear Manifolds}

\begin{document}
\maketitle
\begin{abstract}
\small
\textbf{Abstract} Classical model reduction techniques project the governing equations onto linear subspaces of the high-dimensional state-space. For problems with slowly decaying Kolmogorov-$n$-widths such as certain transport-dominated problems, however, classical linear-subspace reduced-order models (ROMs) of low dimension might yield inaccurate results. Thus, the concept of classical linear-subspace ROMs has to be extended to more general concepts, like Model Order Reduction (MOR) on manifolds. Moreover, as we are dealing with Hamiltonian systems, it is crucial that the underlying symplectic structure is preserved in the reduced model, as otherwise it could become unphysical in the sense that the energy is not conserved or stability properties are lost.
To the best of our knowledge, existing literature addresses either MOR on manifolds or symplectic model reduction for Hamiltonian systems, but not their combination. In this work, we bridge the two aforementioned approaches by providing a novel projection technique called \emph{symplectic manifold Galerkin} (SMG), which projects the Hamiltonian system onto a nonlinear symplectic trial manifold such that the reduced model is again a Hamiltonian system. We derive analytical results such as stability, energy-preservation and a rigorous a-posteriori error bound. 
Moreover, we construct a weakly symplectic deep convolutional autoencoder as a computationally practical approach to approximate a nonlinear symplectic trial manifold. 
Finally, we numerically demonstrate the ability of the method to outperform (non-)structure-preserving linear-subspace ROMs and non-structure-preserving MOR on manifold techniques.
\end{abstract}

\msc{	65P10, 34C20, 37J25, 37M15, 37N30.}

\section{Introduction}
\label{Sec:introduction}
Numerical simulation models for the approximation of complex physical systems are usually of high dimension and, thus, require large computational costs.
This is infeasible whenever real-time evaluations are needed or a multi-query parametric setting is given in which the high-dimensional model needs to be evaluated for various parameters, e.g., for the repeated evaluation in an optimization setting. To accelerate computations, we employ \emph{model order reduction} (MOR) to derive efficiently evaluable low-dimensional models with rigorous error control w.r.t.\ the high-dimensional model. Well-known examples for MOR methods are, e.g., the reduced basis method, the proper orthogonal decomposition or balanced truncation, see \cite{benner2017model} for an introduction to different MOR methods. 

In this paper, our focus is on reducing dynamical systems of ordinary differential equations, which often arise from spatial discretization of partial differential equations. These dynamical systems frequently possess physical characteristics of the modeled phenomena such as an underlying structure or conservation properties. 
Possible structures that have been addressed in the context of MOR comprise, e.g., Lagrangian \cite{lall2003structure, carlberg2015preserving}, Hamiltonian or port-Hamiltonian \cite{polyuga2010structure, chaturantabut2016structure, beattie2019structure} structures.
We are in particular interested in preserving Hamiltonian systems arising in, e.g., plasma physics \cite{morrison1980maxwell} or gyroscopic systems \cite{xu2005numerical}.
To this end we derive a reduced Hamiltonian system which preserves symplecticity, as otherwise the reduced system could become unphysical in the sense that the energy is not conserved or stability properties are lost.
Structure-preserving model reduction for Hamiltonian systems has already been presented for linear-subspace methods in \cite{PM2016,AH2017,BBH2019,BHR2020,UKY2021,sharma2021hamiltonian}, for nonlinear Poisson systems in \cite{HP2021} and in \cite{HPR2020, Pagliantini2021} utilizing a dynamical-in-time approach motivated by low-rank approximations. We extend this literature by developing reduced models for Hamiltonian systems on nonlinear symplectic trial manifolds, which can drastically lower the dimension of the reduced model needed to build an accurate reduced model.

One motivation for using nonlinear manifolds as low-dimensional approximations instead of linear subspaces arises from model reduction of transport-dominated problems. It is known that such problems can exhibit a slow decay of the Kolmogorov-$n$-widths
of, e.g., $n^{-1/2}$ \cite{ohlberger2016reduced, greif2019decay}, which describes the best-possible error for a linear-subspace ROM of size $n$. Recent efforts to break the Kolmogorov-$n$-widths for transport problems include transforming and shifting the basis functions as well as the physical domain of the individual snapshots \cite{iollo2014advection, welper2017interpolation, reiss2018shifted, cagniart2019model, black2020projection}, the method of freezing \cite{ohlberger2013nonlinear} and splitting of one global domain into multiple local linear subspaces \cite{ohlberger2015error}. 
Model reduction for general nonlinear manifolds has already been performed in \cite{gu2011model}, which is limited to piecewise-linear manifolds and in \cite{cruz2020}, where the manifolds is constructed via moment matching and the equations of the nonlinear dynamical systems have to be known. In \cite{Rim2019ManifoldAV}, the authors introduced manifold approximations via transported subspaces, which restricts to 1D problems for now. 
Recent approaches using autoencoders have been introduced in \cite{LC2020A} as MOR on manifolds and extended in \cite{LC2020B} to satisfy physical conservation laws. In \cite{KCWZ2020B}, a shallow masked autoencoder combined with a hyper-reduction approach has been employed to construct a reduced manifold. In \cite{FDM2021}, the authors not only learn a reduced manifold via an autoencoder, but also utilize a feed-forward network to approximate the reduced model.

To the best of our knowledge the aforementioned approaches either consider symplectic model reduction or MOR on manifolds, but not their combination. Our goal is to bridge the two approaches by providing a method for symplectic model reduction on manifolds. 
In particular, this article provides the following new contributions:
\begin{enumerate}
	\item Symplectic manifold Galerkin (SMG) projection techniques, which project Hamiltonian systems on nonlinear symplectic trial manifolds (\cref{Sec:nonlinear_trial_manifold}). Using this technique, the resulting SMG-ROM is again a Hamiltonian system. 
	\item Analysis, which includes:\\[-1.4em]
	\begin{itemize}
		\item Preservation of energy and stability: we detail sufficient conditions for the reduced model to preserving energy and/or stability in the sense of Lyapunov (\cref{Sec:energy_preservation,Sec:stability_preservation}).\\[-1.4em]
		\item Error bound: we develop a rigorous a-posteriori error bound based on a time-discrete formulation (\cref{Sec:error_bound}).
	\end{itemize}
	\item Weakly symplectic deep convolutional autoencoder (weakly symplectic DCA): we present a new weakly symplectic DCA (\cref{Sec:SymAuto}) to approximate a nonlinear symplectic trial \mbox{manifold.}\footnote{Note, that we are not restricted to autoencoders here, but could use any technique which provides a symplectic map from the low-dimensional to the high-dimensional space, see \cref{Rem:AlternativesToAutoencoders}.}%
	\item Numerical experiments: We numerically investigate the linear wave equation for the challenging case of a thin moving pulse and demonstrate the ability of the SMG-ROM to outperform (a) (non-)structure-preserving linear-subspace ROMs and (b) non-structure-preserving MOR on manifolds (\cref{Sec:exp_lin_wave}).
\end{enumerate}

This paper is structured as follows: 
in \cref{Sec:SMG}, we define the high-dimensional Hamiltonian system and subsequently construct the reduced Hamiltonian system on a nonlinear symplectic trial manifold. Additionally, we show that the reduced model can inherit properties like energy- and stability-preservation and we develop a rigorous a-posteriori error bound. In \cref{Sec:SymAuto}, we detail the weakly symplectic DCA. Numerical experiments on a Hamiltonian system with known slowly decaying Kolmogorov-$n$-widths are performed in \cref{Sec:NumExp}. Finally, we conclude in \cref{Sec:conclusions}. 

\section{Symplectic Model Reduction on Manifolds}
\label[secname]{Sec:SMG}
In this section, we introduce the model reduction of Hamiltonian systems on nonlinear symplectic trial manifolds. To this end, we start with a review of the necessary details on symplectic geometry and subsequently state our full-order model, a parametric Hamiltonian system. Then, we detail our reduced model on a nonlinear symplectic trial manifold and show that stability and energy can be preserved with our \emph{symplectic manifold Galerkin} (SMG) projection technique. We conclude this section by developing a rigorous a-posteriori error bound.

\subsection{Symplectic Geometry}
Let $\cM$ be a smooth manifold and let $\omega$ be a symplectic form on $\cM$, i.e., a closed, non-degenerate differential 2-form.\footnote{See, e.g., \cite{marsden1995introduction, lee2013} for an introduction of smooth manifolds.} It follows immediately, that $\cM$ needs to be of even dimension \cite[Proposition 22.7]{lee2013}.
The combination of a smooth manifold and a symplectic form $(\cM, \omega)$ is denoted a symplectic manifold.
For the case of $\cM$ being a vector space over $\R$ of even dimension $2N$, $N \in \mathbb{N}$, we introduce the corresponding definition of a symplectic form. 
\begin{mydef}[Symplectic Form over $\R$, Symplectic Vector Space]
Let $\mathbb{V}$ be an $\R$-vector space of dimension $2N$, $N \in \mathbb{N}$. Let $\omega: \mathbb{V} \times \mathbb{V} \rightarrow \R$ be a
skew-symmetric and non-degenerate bilinear form, i.e., for all $v_1, v_2 \in \mathbb{V}$, it holds
\begin{align*} 
\omega (v_1,v_2) = - \omega (v_2, v_1) \qquad \textnormal{and} \qquad \omega(v_1, v_3) = 0 \quad \forall v_3 \in \mathbb{V} \ \Rightarrow  \  v_1 = 0.
\end{align*}
The bilinear form $\omega$ is called \emph{symplectic form} on $\mathbb{V}$ and the pair $(\mathbb{V}, \omega)$ is called \emph{symplectic vector space}.
\end{mydef}
For a symplectic form $\omega$ on $\mathbb{V}$, dim$(\mathbb{V})=2N$, it is known that there exists a basis\\
\mbox{$\{e_1,\ldots, e_N, f_1, \ldots, f_N \}\subset \mathbb{V}$} such that the symplectic form can be written in \emph{canonical coordinates}
\begin{align}\label{Eq:Canonical_Sympl_Form}
\omega(v_1,v_2)
= \omega_{2N}(\bfv_1, \bfv_2)
:= \rT{\bfv_1} \J_{2N} \bfv_2, \qquad \forall v_1, v_2 \in \mathbb{V},
\end{align}
where $\bfv_1, \bfv_2 \in \R^{2N}$ denote the respective coordinate vector of $v_1, v_2$, see, e.g., \cite{da2008lectures}.
The matrix
\begin{align*}
\J_{2N} := \begin{pmatrix} \bm{0}_N & \mathbb{I}_N \\ -\mathbb{I}_N  & \bm{0}_N \end{pmatrix} \in \R^{2N \times 2N}, 
\end{align*}
is called the \emph{canonical Poisson matrix}, where $\mathbb{I}_N\in \R^{N \times N}$ is the identity matrix and $\bm{0}_N \in \R^{N \times N}$ is the matrix of all zeros. One can easily see, that $\rT{\J_{2N}} \J_{2N} = \J_{2N} \rT{\J_{2N}} = \mathbb{I}_{2N}$ and $\J_{2N} \J_{2N}= \rT{\J_{2N}}\rT{\J_{2N}} = -\mathbb{I}_{2N}$.
With this notation at hand, we define a symplectic map as follows. 
\begin{mydef}[Symplectic Map, Symplectic Matrix {\cite[Section VI.2]{HLW2006}}]
Let $U\subset \R^{2n}$ be an open set. A differentiable map $d:U \rightarrow \R^{2N} (n \le N)$ is called \emph{symplectic} if the Jacobian matrix\footnote{Note that the subscript $\bfx$ denotes the derivative w.r.t.\ the variable and the function argument $\bfx$ indicates the position at which the Jacobian is evaluated.} $\Jacx{\bfx} \in \R^{2N\times 2n}$ is a \emph{symplectic matrix} for every $\bfx \in U$, i.e.,
\begin{align}\label{Eq:Canonical_sympl}
\rTb{\Jacx{\bfx}} \J_{2N} \Jacx{\bfx} = \J_{2n}.
\end{align}
\end{mydef}
Note that due to the above definition, if the Jacobian $\Jacx{\bfx}$ is a symplectic matrix, it implies that $\Jacx{\bfx}$ has full column rank, i.e., rank$(\Jacx{\bfx})=2n$ for $n \le N$, since $\J_{2N}$ and $\J_{2n}$ are nonsingular. 
With the above definition, we now can introduce the notion of the symplectic inverse. 

\begin{mydef}[Symplectic Inverse]
If $\Jacx{\bfx} \in \R^{2N\times 2n}$ is a symplectic matrix, we denote by\\
$\sI{\Jacx{\bfx}} \in \R^{2n\times 2N}$ the \emph{symplectic inverse}\footnote{The symplectic inverse should not be confused with the Moore--Penrose pseudo inverse which is often referred with a similar notation.}
\begin{align*}
\sI{\Jacx{\bfx}} :=  \rT{\J_{2n}} \rTb{\Jacx{\bfx}}\J_{2N}.
\end{align*}
\end{mydef}
It follows that 
\begin{align}\label{eq:prop_symplectic_inverse}
\sI{\Jacx{\bfx}}{\Jacx{\bfx}}=   \rT{\J_{2n}} \rTb{\Jacx{\bfx}}\J_{2N} {\Jacx{\bfx}}= \rT{\J_{2n}} \J_{2n} = \mathbb{I}_{2n},
\end{align}
as well as
\begin{align}\label{Eq:prop_sym_matrix}
\sI{\Jacx{\bfx}} \J_{2N} =  - \rT{\J_{2n}} \rTb{\Jacx{\bfx}}=  \J_{2n} \rTb{\Jacx{\bfx}} .
\end{align}

\subsection{Full-Order Model} We introduce our high-dimensional full-order model (FOM), which is a parametric autonomous Hamiltonian system. Let $\cH:\R^{2N} \times \cP \rightarrow \R$ be a smooth function called \emph{Hamiltonian function} dependent on the parameter $\bfmu \in \cP \subset \R^{p}, p \in \mathbb{N}$. Further, let $\mathcal{I}:=(0,T]$ be a time interval with final time $0 < T < \infty$. The FOM in canonical coordinates reads: for a given $\bfmu \in \cP$, $\bfx_0(\bfmu) \in \R^{2N}$, find $\bfx(\cdot;\bfmu) \in C^{1}(\mathcal{I}, \R^{2N})$ such that
\begin{align}\label{Eq:FOM_canonical}
\ddt \bfx(t;\bfmu) = \J_{2N} \nabla_{\bfx}\cH(\bfx(t;\bfmu);\bfmu) =: X_{\cH}(\bfx(t;\bfmu);\bfmu), \quad \bfx(0;\bfmu) = \bfx_{0}(\bfmu),
\end{align}
where $X_{\cH}:\R^{2N} \times \cP \rightarrow \R$ is the so-called \emph{Hamiltonian vector field} and $C^{1}(\mathcal{I}, \R^{2N})$ denotes continuously differentiable functions in time with values in $\R^{2N}$.
The \emph{flow} $\phi_t(\cdot; \bfmu): \R^{2N} \rightarrow \R^{2N}$ of $X_{\cH}$ is a mapping evolving the initial state to the corresponding solution of \cref{Eq:FOM_canonical}, i.e.,
\begin{align*}
\phi_t(\bfx_0(\bfmu);\bfmu) := \bfx(t;\bfmu; \bfx_0(\bfmu)),
\end{align*}
where $\bfx(t;\bfmu; \bfx_0)$ represents the solution with initial data $\bfx_0(\bfmu)$.
An important property of the flow is that it preserves the Hamiltonian $\cH$
\begin{align*}
\ddt \cH(\phi_t(\bfx;\bfmu);\bfmu) &=  \rTb{\nabla_{\bfx}\cH(\phi_t(\bfx;\bfmu);\bfmu)}\ddt \phi_t(\bfx;\bfmu) \\&=  \rTb{\nabla_{\bfx}\cH(\phi_t(\bfx;\bfmu);\bfmu)}  \J_{2N} \nabla_{\bfx}\cH(\phi_t(\bfx;\bfmu);\bfmu)  =0,
\end{align*}
where the last equality follows from the skew-symmetry of $\J_{2N}$. 

In order that \cref{Eq:FOM_canonical} is well-posed, we assume that for any $\bfmu \in \cP$ the operator $X_{\cH}(\cdot;\bfmu): \R^{2N} \rightarrow \R$ is Lipschitz continuous w.r.t.\ $\bfx$. 

\subsection{Symplectic Manifold Galerkin ROM}
\label{Sec:nonlinear_trial_manifold}

We now detail how to derive a symplectic manifold Galerkin ROM (SMG-ROM) on a nonlinear symplectic trial manifold and at the same time ensure that this reduced system preserves the symplectic structure of the Hamiltonian system. This nonlinear symplectic trial manifold aims at approximating the detailed solution manifold 
\begin{align*}
\cM := \left\{ \bfx(t;\bfmu) \in \R^{2N} \;\big|\; t \in [0,T], \bfmu \in \cP \right\}. 
\end{align*}
In particular, we are looking for approximate solutions of the form 
\begin{align} \label{Eq:reconstructed_solution} 
 	\tbfx(t;\bfmu) := \bfxref(\bfmu) + d(\bfx_r(t;\bfmu)) \approx \bfx(t;\bfmu),
\end{align}
with the (possibly parametric) \emph{reference state} $\bfxref(\bfmu) \in \R^{2N}$, the \emph{reduced coordinates} $\bfx_r(t;\bfmu) \in \R^{2n}$ (also known as \emph{latent variables}) with $n \ll N$, the \textit{reconstructed solution} $\tbfx(t;\bfmu) \in \R^{2N}$ and the \emph{reconstruction function} (or \emph{decoder}) $d: \R^{2n} \rightarrow \R^{2N}$.
The reconstruction function $d$ is assumed to be a (mostly nonlinear) symplectic map \cref{Eq:Canonical_sympl}.
With this choice, the reconstructed solution evolves on the nonlinear trial manifold
\begin{align}\label{Eq:trial_manifold}
\tcM(\bfmu) := \bfxref(\bfmu) + \left\{ d(\bfa_r) \in \R^{2N} \;\big|\; \bfa_r \in \R^{2n} \right\}.
\end{align}
The evolution of the reconstructed solution in time is 
\begin{align} \label{Eq:deriv_reconstructed}
\ddt \tbfx(t;\bfmu) =  \left(\bm{D}_{\bfx_r} d(\bfx_r(t;\bfmu))\right)  \ddt \bfx_r(t;\bfmu),
\end{align}
due to the chain rule. To construct the ROM, we define the time-continuous residual of the FOM \cref{Eq:FOM_canonical} w.r.t.\ the approximation $\tbfx(t;\bfmu)$ by 
\begin{align*}
 r\left(t;\bfmu\right) &:= \ddt \tbfx(t;\bfmu) - \J_{2N} \nabla_{\bfx}\cH(\tbfx(t;\bfmu);\bfmu) \\
 &=  \left(\bm{D}_{\bfx_r} d(\bfx_r(t;\bfmu))\right)  \ddt \bfx_r(t;\bfmu)- \J_{2N} \nabla_{\bfx}\cH(\bfxref(\bfmu) + d(\bfx_r(t;\bfmu));\bfmu).
\end{align*}
For the ROM, we require that the symplectic projection of the residual
\begin{align*}
 \left(\bm{D}_{\bfx_r} d(\bfx_r(t;\bfmu))\right)^{+}r&\left(t;\bfmu\right)
\stackrel{!}{=} \bfzero_{2n},
\end{align*}
vanishes. Note that $\left(\bm{D}_{\bfx_r} d(\bfx_r(t;\bfmu))\right)^{+}$ exists due to the symplecticity of $d$. Reformulating this equation yields
\begin{align*}
\ddt \bfx_r(t;\bfmu) &= \left(\bm{D}_{\bfx_r} d(\bfx_r(t;\bfmu))\right)^{+} \J_{2N} \nabla_{\bfx}\cH(\bfxref(\bfmu) + d(\bfx_r(t;\bfmu));\bfmu)\\
&= \J_{2n}
\underbrace{
	\rT{\left(\bm{D}_{\bfx_r} d(\bfx_r(t;\bfmu))\right)} \nabla_{\bfx}{\cH}(\bfxref(\bfmu) + d(\bfx_r(t;\bfmu));\bfmu)
}_{
	= \nabla_{\bfx_r}{\cH_r}(\bfx_r(t;\bfmu);\bfmu)
},
\end{align*}
which follows due to the properties of the Poisson matrix in \cref{Eq:prop_sym_matrix}, \cref{eq:prop_symplectic_inverse} and the chain rule. We denote by ${\cH_r}(\cdot; \bfmu): \R^{2n} \rightarrow \R$, $\bfx_r \mapsto \cH(\bfxref(\bfmu) + d(\bfx_r); \bfmu)$, the \emph{reduced Hamiltonian}. 
Thus, the SMG-ROM on the nonlinear symplectic trial manifold reads:
 for a given $\bfmu \in \cP$, $\bfx_{r,0}(\bfmu) \in \R^{2n}$, find $\bfx_r(\cdot;\bfmu) \in C^{1}(\mathcal{I}, \R^{2n})$ such that
\begin{align}\label{Eq:ROM_canonical}
&\ddt \bfx_r(t;\bfmu) = \J_{2n} \nabla_{\bfx_r}{\cH_r}(\bfx_r(t;\bfmu);\bfmu)&
&\bfx_{r}(0;\bfmu) = \bfx_{r,0}(\bfmu).
\end{align}
By $X_{{\cH_r}}:\R^{2n} \times \cP \rightarrow \R^{2n}$, $(\bfx_r,\bfmu) \mapsto \J_{2n} \nabla_{\bfx_r}{\cH_r}(\bfx_r;\bfmu)$ we denote the reduced Hamiltonian vector field.
Moreover, the choice of $d$ to be a symplectic map guarantees that the tuple $(\tcM(\bfmu), \omega_{2N})$ with the nonlinear trial manifold from \cref{Eq:trial_manifold} and the canonical symplectic form from \cref{Eq:Canonical_Sympl_Form} is a symplectic manifold, see {\cite[Section VII.2.3]{HLW2006}}. This is why we refer to $(\tcM(\bfmu), \omega_{2N})$ as a \emph{nonlinear symplectic trial manifold}.

In order to solve \cref{Eq:ROM_canonical} numerically, a symplectic numerical integration scheme is applied and the resulting system of nonlinear equations is solved for each discrete time step with a quasi-Newton scheme which omits second-order derivatives of $d$ in analogy to \cite[Section 3.4]{LC2020A}.

In the following we show that energy and stability can be preserved during the time evolution, since the above constructed reduced system is again a Hamiltonian system.

\subsection{Energy Preservation} 
\label{Sec:energy_preservation}

\begin{myprop}\label{prop:err_in_hamiltonian}
Let $\bfx(t;\bfmu)$ be the solution of the FOM \cref{Eq:FOM_canonical} at time $t$ and let $\tbfx(t;\bfmu)$ be the reconstructed solution \cref{Eq:reconstructed_solution} obtained from solving a SMG-ROM. Then, for any $\bfmu \in \cP$, the error in the Hamiltonian
\begin{align} \label{eq:err_in_hamiltonian}
	\varDelta \cH(t;\bfmu)
:= |\cH(\bfx(t;\bfmu);\bfmu) - \cH(\tbfx(t;\bfmu);\bfmu)|
\equiv |\cH(\bfx_0(\bfmu);\bfmu) - \cH(\tbfx(0, \bfmu);\bfmu)| 
\end{align}
is constant for all $t \in \cI$.
\end{myprop}
\begin{proof}
The proof is based on the property that the Hamiltonian is preserved over time for both, FOM and ROM. Omitting $\bfmu$-dependency for readability, this reads
\begin{align*}
&\cH(\bfx(t))
= \cH(\bfx_0)&
&\text{and}&
&\cH(\tbfx(t))
= {\cH}_r(\bfx_r(t)) 
= {\cH}_r(\bfx_r(0))
= \cH(\tbfx(0)).&
\end{align*}
\end{proof}

\begin{myrem} \label{rem:exact_reproduction}
The reference state $\bfxref(\bfmu)$ is set to $\bfx_0(\bfmu) - d(\bfx_{r,0}(\bfmu))$ for an arbitrary initial value $\bfx_{r,0}(\bfmu) \in \R^{2n}$ of the reduced system. As observed in \cite[Remark 3.1]{LC2020A}, this leads to an exact reproduction of the initial value $\tbfx(0; \bfmu) = \bfx_0(\bfmu)$. With the SMG projection, we can further conclude with \cref{prop:err_in_hamiltonian} that this implies an exact reproduction of the Hamiltonian $\varDelta \cH (t; \bfmu) \equiv 0$.
\end{myrem}

\subsection{Stability Preservation}
\label{Sec:stability_preservation}
In this section we examine the question of stability of the full-order and reduced model in the sense of Lyapunov. For the ease of notation, we omit the time and $\bfmu$-dependency for the remainder of this section whenever possible.

Consider the differential equation 
\begin{align} \label{Eq:dyn_system}
\ddt{\bm{x}} = f(\bm{x}),
\end{align}
where $f$ is a smooth function from an open set $U \subset \R^{2N}$ into $\R^{2N}$. Let $f$ have an equilibrium point $\bm{x}_e \in U$, such that $f(\bm{x}_e)=0$ and let $\phi(t, \bfx)$ be the flow of \cref{Eq:dyn_system}. Then, Lyapunov stability of the point $\bm{x}_e$ is defined as: 
\begin{mydef}[Lyapunov Stability {\cite[Section 12]{MO2017}}]
An equilibrium point $\bm{x}_e \in U$ of \cref{Eq:dyn_system} is \emph{Lyapunov stable} if for every $\epsilon >0$, there exists $\delta >0$ such that $\norm{\bm{x}_e - \phi(t,\bm{x}_0)}_2<\epsilon$ for all $0 \le t< \infty$ whenever $\norm{ \bm{x}_e - \bm{x}_0 }_2<\delta$. 
\end{mydef}

In order to state sufficient conditions for an equilibrium point to be Lyapunov stable, we recall Lyapunov's stability theorem. 

\begin{mythm}[Lyapunov's Stability Theorem {\cite[Theorem 12.1.1]{MO2017}}] \label{Thm:Lyapunov_stability}
Let $\bm{x}_e$ be an equilibrium point of \cref{Eq:dyn_system}. If there exists a function $V:U\rightarrow \R$ with $\bm{x}_e \in U\subset \R^{2N}$, such that $\nabla_{\bfx} V(\bm{x}_e)=\bm{0}$, $\nabla_{\bfx}^{2} V(\bm{x}_e)$ positive definite (where $\nabla_{\bfx}^2 V$ denotes the Hessian of $V$) and 
\begin{align} \label{Eq:cond_Lyapunov}
\rT{\nabla_{\bfx} V(\bm{x})}f(\bm{x}) \le {0},\quad  \forall \bm{x} \in U \subset \R^{2N},
\end{align}
then $\bm{x}_e$ is called a Lyapunov stable point. 
\end{mythm}

The function $V$ in the aforementioned theorem is also called the \emph{Lyapunov function}. For Hamiltonian systems, a suitable candidate for $V$ is the Hamiltonian itself. In the case of autonomous Hamiltonians considered in this paper, a point $\bm{x}\in \R^{2N}$ is an equilibrium point of \cref{Eq:FOM_canonical} if and only if $\bm{x}$ is a critical point of $\cH$, i.e., a point s.t. $\nabla_{\bfx} \cH (\bm{x})=\bm{0}$ \cite[Proposition 3.4.16]{AM1987}. Following the well-known Dirichlet's stability theorem for Hamiltonian systems \cite[Corollary 12.1.1]{MO2017}, it suffices that the equilibrium point is a strict local maximum or minimum in order to be a Lyapunov stable point.  \Cref{Eq:cond_Lyapunov} is automatically fulfilled for autonomous Hamiltonian systems which can be easily seen by 
\begin{align*}
\rT{\nabla_{\bfx} \cH(\bm{x})} f(\bfx)= \rT{\nabla_{\bfx} \cH(\bm{x})} \J_{2N} \nabla_{\bfx} \cH(\bm{x}) = 0, \quad \bm{x} \in U,
\end{align*}
due to the skew-symmetry of $ \J_{2N} $. Similar to \cite[Theorem 18]{AH2017} in the special case of $d$ linear, we are now able to show, that if $\cH$ (or $-\cH$) is a Lyapunov function in an environment of the equilibrium point, then these equilibrium points for the full-order and reduced model are Lyapunov stable. In the following we extend this theorem to hold for more general nonlinear $d$.

\begin{mythm}[Lyapunov Stability]
Assume that $\bfx_e \in \R^{2N}$ is an equilibrium point of \cref{Eq:FOM_canonical} and there exists $\bfx_{r,e}\in \R^{2n}$ such that $\bfx_e = \bfxref + d(\bfx_{r,e})$.
If $\cH$ (or $-\cH$) is a Lyapunov function as defined in \cref{Thm:Lyapunov_stability}, then $\bfx_e$ and $\bfx_{r,e}$ are Lyapunov stable equilibrium points for the full-order model \cref{Eq:FOM_canonical} and the reduced model \cref{Eq:ROM_canonical}, respectively.
\end{mythm}
\begin{proof}
As $\bfx_e$ is an equilibrium point and $\cH$ is a Lyapunov function, it immediately follows from \cref{Thm:Lyapunov_stability} that $\bfx_e$ is Lyapunov stable. 
By using the chain rule, we arrive at $\nabla_{\bfx_r} \cH_r(\bfx_r) = \rTb{\Jac{\bfx_r}} \nabla_{\bfx} \cH(\bfxref + d(\bfx_r))$. Evaluating the latter at $\bfx_{r,e}$ yields
\begin{align*}
\nabla_{\bfx_r} \cH_r(\bfx_{r,e}) = \rT{\left(\bm{D}_{\bfx_r} d(\bfx_{r,e}) \right)}\nabla_{\bfx} \cH(\bfxref + d(\bfx_{r,e})) = \rT{\left(\bm{D}_{\bfx_r} d(\bfx_{r,e}) \right)}\nabla_{\bfx} \cH(\bfx_{e}) = \bm{0},
\end{align*}
where the last equivalence follows as $\bfx_e$ is a critical point for $\cH$, i.e., $\nabla_{\bfx}\cH(\bfx_{e}) = \bm{0}$. 
Thus $\bfx_{r,e}$ is an equilibrium point for \cref{Eq:ROM_canonical}. In order to show that $\bfx_{r,e}$ is a strict local minimum, we consider the Hessian matrix. Again, considering the chain rule, we get
\begin{align*}
\nabla_{\bfx_{r}}^2\cH_r(\bfx_{r,e}) =  \rTb{\bm{D}_{\bfx_r} d(\bfx_{r,e})} \nabla^2_{\bfx} \cH(\bfxref + d(\bfx_{r,e}))  (\bm{D}_{\bfx_r} d(\bfx_{r,e})).
\end{align*}
Therefore, for any $\bm{\xi} \in \R^{2n} \backslash \{ \bm{0} \} $, we obtain that $\nabla_{\bfx_r}^{2}\cH_r(\bfx_{r,e})$ is positive definite
\begin{align*}
\rT{\bm{\xi}}  \nabla_{\bfx_r}^2\cH_r(\bfx_{r,e}) \bm{\xi}  =\rT{\left(\bm{D}_{\bfx_r} d(\bfx_{r,e})\bm{\xi} \right)} \nabla^2_{\bfx} \cH(\bfxref + d(\bfx_{r,e})) \left(\bm{D}_{\bfx_r} d(\bfx_{r,e})\bm{\xi} \right) >0,
\end{align*}
due to the positive definiteness of $\nabla_{\bfx}^2 \cH(\bfx_e)$. With Dirichlet's stability theorem, we conclude that $\bfx_{r,e}$ is a Lyapunov stable point for \cref{Eq:ROM_canonical}.
\end{proof}

\subsection{Error Estimation}
In this section, we deduce a rigorous a-posteriori error bound. This error bound is coupled to the chosen time-discretization scheme, which needs to be a so-called symplectic integrator \cite{HLW2006, BM2017} in order to preserve the underlying symplectic structure of the Hamiltonian system. 
 
\subsubsection{Time Discretization}
We use a Runge-Kutta (RK) scheme for the time discretization of our Hamiltonian system. Since not all RK methods meet the requirements of being a symplectic integrator, we detail the necessary conditions after the definition.

We start by applying an equidistant time discretization to the interval $(0,T]$ for $0 < T < \infty$. Therefore, we set $\varDelta t:= \frac{T}{K}, K \in \mathbb{N}, t^{k}:= k\varDelta t, 0 \le k \le K$, and seek an approximation $\bfx^k \approx \bfx(t^{k})$. Note that the superscript, against the standard notation, denotes the enumeration of time-dependent variables instead of the power of those quantities.

\begin{mydef}[Runge Kutta method]
An s-stage RK method (formulated for the autonomous Hamiltonian system \cref{Eq:FOM_canonical}) with $s \in \mathbb{N}$, real numbers $b_i$, $a_{ij}$ $(i,j =1,\ldots,s)$ is for $1\le k \le K$ given by 
\begin{align*}
	\bfw_i^k &= X_{\cH}\left(\bm{x}^{k-1}+ \varDelta t \sum_{j=1}^{s}a_{ij}\bfw_j^k \right), \qquad i = 1, \ldots ,s, \\
	\bm{x}^{k} &= \bm{x}^{k-1} + \varDelta t\sum_{i=1}^{s} b_i \bfw_i^k.
\end{align*}
\end{mydef}
Runge-Kutta methods are symplectic if the coefficients satisfy
\begin{align*}
	b_i a_{ij} + b_j a_{ji}=b_i b_j, \qquad \textnormal{for all } i,j = 1, \ldots, s, 
\end{align*}
see \cite[Theorem VI.4.3, p.178]{HLW2006}, which, e.g., holds for the implicit midpoint rule, 
\begin{subequations} 
\begin{align} \label{eq:impl_midpoint}
	\bfw_1^k &= X_{\cH}\left(\bm{x}^{k-1}+ \frac{\varDelta t}{2} \bfw_1^k \right), \\
	\bm{x}^{k} &= \bm{x}^{k-1} + \varDelta t \bfw_1^k, 
\end{align}
\end{subequations}
which therefore is a symplectic integrator. 
To compute a solution for the time-discrete ODE, we need to solve the following algebraic equations at every discrete time-step $k$ 
\begin{align} \label{Eq:FOM_RK}
\bm{r}_{i}^{k}(\bm{w}^{k}_1, \ldots, \bfw^{k}_s) = \bfzero,  \qquad i =1, \ldots,s,
\end{align}
where the $i-$th RK residual is defined by
\begin{align*}
\bfr_{i}^{k}(\bfw_1, \ldots, \bfw_s) := \bfw_i - X_{\cH}\bigg(\bfx^{k-1} + \varDelta t \sum_{j=1}^{s}a_{ij} \bfw_j \bigg), \qquad i =1, \ldots,s,
\end{align*}
with the updated state 
\begin{align}\label{Eq:state_FOM}
\bfx^{k} = \bfx^{k-1} +\varDelta t \sum_{i=1}^{s} b_i \bfw_i^{k}.
\end{align}

In order to compute a time-discrete approximation of the reduced problem \cref{Eq:ROM_canonical}, we solve the algebraic equations for $\bm{w}^{k}_{r,i} \in \R^{2n}, i=1, \ldots,s$
\begin{align} \label{Eq:ROM_RK}
\bm{r}_{r,i}^{k}(\bm{w}^{k}_{r,1}, \ldots, \bfw^{k}_{r,s}) = \bfzero,  \qquad i =1, \ldots,s,
\end{align}
where the $i-$th reduced RK residual and updated reduced state are given by
\begin{subequations}
\begin{align}
\bm{r}_{r,i}^{k}(\bfw_{r,1}, \ldots, \bfw_{r,s}) :&= \bfw_{r,i} - X_{\cH_r}\bigg(\bfx_r^{k-1} + \varDelta t \sum_{j=1}^{s}a_{ij} \bfw_{r,j} \bigg),&
	&i =1, \ldots,s, \\
\bfx_{r}^{k} &= \bfx_r^{k-1} +\varDelta t \sum_{i=1}^{s} b_i \bfw_{r,i}^{k}.  \label{Eq:Update_ROM}
\end{align}
\end{subequations}

\subsubsection{Error Bound}
\label{Sec:error_bound}
In this section we provide a rigorous a-posteriori error bound. We assume that a RK method, defined in the section above, has been used to construct the time-discrete scheme. For linear $d$, bounds have been derived in \cite[Theorem 6.19]{CBA2017}, which we extend in the following to error bounds for more general nonlinear $d$. 

\begin{mythm} [Error Bound] 
If $X_{\cH}$ is Lipschitz continuous, i.e., there exists a constant $\kappa >0$ such that 
\begin{align*}
\| X_{\cH}(\bfx) - X_{\cH}(\bm{y}) \|_{2} \le \kappa \| \bfx - \bm{y}  \|_2, \qquad \forall \bfx, \bm{y} \in \R^{2N},
\end{align*}
and $\varDelta t$ is chosen sufficiently small such that
\begin{enumerate}[label=$(a\arabic*)$]
	\item the matrix $\bm{D} \in \R^{s \times s}$ with entries $d_{ij}:= \delta_{ij} - \kappa \varDelta t |a_{ij}|$ is invertible and
	\item for every $\bfx_s, \bm{y}_s \in \R^{s}_{\ge \bm{0}}$, if $\bm{D}\bfx_{s} \le \bm{y}_s$, then $\bfx_s \le \bm{D}^{-1}\bm{y}_s$,\footnote{For the implicit midpoint rule $(s=1)$, this condition holds if $D=1- (\kappa \varDelta t)/{2} \in \R^{+}$, which can be ensured by choosing $\varDelta t < 2/\kappa$.}
\end{enumerate}
then the following error bound holds for $1 \le k \le K$
\begin{align*}
\| \bfx^{k} &- \tbfx^{k}   \|_{2} \le (c_1)^{k} \norm{\bfx^{0} - \tbfx^{0} }_2 +  \sum_{\ell =0}^{k-1} (c_1)^{\ell} \cdot \\
&\left(\varDelta t \sum_{m=1}^{s}|b_m|\sum_{i=1}^{s}[\bm{D}^{-1}]_{mi}  {\|\tbfr^{k-\ell}_i(\tbfw_1^{k- \ell}, \ldots, \tbfw_s^{k-\ell}) \|_{2}} +\norm{\tbfr_{\bfx}^{k-\ell}(\tbfx^{k-\ell}) }_2\right),
\end{align*}
with 
\begin{align*}
c_1 :=  \left( 1+ \kappa \varDelta t \sum_{m=1}^{s}|b_m|\sum_{i=1}^{s}[\bm{D}^{-1}]_{mi} \right).
\end{align*}
Here, $\bfx^{k}$ respective $\bfx^{k}_r$ is a solution of \cref{Eq:FOM_RK}, \cref{Eq:ROM_RK} in time instance $k$ with update in \cref{Eq:state_FOM}, \cref{Eq:Update_ROM}. By $\tbfx^{k} := \bfxref + d(\bfx_r^{k})$ we denote the reconstructed state and by $\tbfw_i^{k} := \bm{D}_{\bfx_r} d(\bfx_r^{k-1})\bfw_{r,i}^{k},\  i = 1, \ldots, s$ the reconstructed velocities. The residuals w.r.t.\ the reconstructed state and velocities $\tbfr^{k}_i, \tbfr^{k}_{\bfx},$ are defined by
\begin{subequations}\label{Eq:RK_residual_ROM}
\begin{align}
\tbfr_{i}^{k}({\bfw}_1, \ldots, {\bfw}_s) :&= {\bfw}_i - X_{\cH}\bigg(\tbfx^{k-1} + \varDelta t \sum_{j=1}^{s}a_{ij} {\bfw}_j \bigg), \qquad i =1, \ldots,s, \\ 
\tbfr_{\bfx}^{k}({\bfx}) :&= {\bfx} - \tbfx^{k-1} -\varDelta t \sum_{i=1}^{s} b_i \tbfw_i^{k}.\label{Eq:state_red}
\end{align}
\end{subequations}
Note that if the mapping $d$ is linear, we have the standard Runge-Kutta update step \cref{Eq:state_FOM} for the reconstructed solution.
\end{mythm}

\begin{proof}
Subtracting $\tbfr_{i}^{k}(\tbfw_1^{k}, \ldots, \tbfw_s^{k})$ from $\bfr_{i}^{k}(\bfw_1^{k}, \ldots, \bfw_s^{k})$ yields

\begin{align*}
\bfr_{i}^{k}(\bfw_1^{k}, \ldots, \bfw_s^{k}) &- \tbfr_{i}^{k}(\tbfw_1^{k}, \ldots, \tbfw_s^{k}) = \bfw_i^{k} - \tbfw_i^{k} \\&- \left(  X_{\cH}\bigg(\bfx^{k-1} + \varDelta t \sum_{j=1}^{s}a_{ij} \bfw_j^{k}\bigg) -  X_{\cH}\bigg(\tbfx^{k-1} + \varDelta t \sum_{j=1}^{s}a_{ij} \tbfw^{k}_j\bigg) \right).
\end{align*}
Noting that $\bfr_{i}^{k}(\bfw_1^{k}, \ldots, \bfw_s^{k}) = 0$, we arrive at 
\begin{align*}
\bfw_i^{k} - \tbfw^{k}_i &= - \tbfr_{i}^{k}(\tbfw^{k}_1, \ldots, \tbfw^{k}_s) \\&+ \left(  X_{\cH}\bigg(\bfx^{k-1} + \varDelta t \sum_{j=1}^{s}a_{ij} \bfw_j^{k} \bigg) -  X_{\cH}\bigg(\tbfx^{k-1} + \varDelta t \sum_{j=1}^{s}a_{ij} \tbfw^{k}_j \bigg) \right).
\end{align*}
Taking the norm on both sides and using the Lipschitz continuity yields
\begin{align*}
\| \bfw_i^{k} - \tbfw^{k}_i \|_2 \le \| \tbfr_{i}^{k}(\tbfw^{k}_1, \ldots, \tbfw^{k}_s)  \|_2 + \kappa \| \bfx^{k-1} - \tbfx^{k-1} \|_2 + \kappa \varDelta t \sum_{j=1}^{s}|a_{ij}|  \| \bfw_j^{k} - \tbfw^{k}_j \|_2. 
\end{align*}
Rearranging both sides leads to
\begin{align*}
\| \bfw_i^{k} - \tbfw^{k}_i \|_2 - \kappa \varDelta t \sum_{j=1}^{s}|a_{ij}|  \| \bfw_j^{k} - \tbfw_j^{k} \|_2 \le \norm{\tbfr_{i}^{k}(\tbfw^{k}_1, \ldots, \tbfw^{k}_s) }_2 + \kappa \| \bfx^{k-1} - \tbfx^{k-1} \|_2. 
\end{align*}
As $\varDelta t $ is small enough s.t. assumptions (a1), (a2) hold, for $m=1, \ldots, s$, we arrive at
\begin{align}\label{Eq:upper_bound}
\| \bfw_m^{k}- \tbfw_m^{k} \|_2 \le \sum_{i=1}^{s} [\bm{D}^{-1}]_{mi} \Big( \| \tbfr_{i}^{k}(\tbfw_1^{k}, \ldots, \tbfw^{k}_s)  \|_2 + \kappa  \| \bfx^{k-1} - \tbfx^{k-1} \|_2 \Big).
\end{align}
Evaluating the residual \cref{Eq:state_red} at $\tbfx^{k}$ and adding a zero by reformulating \cref{Eq:state_FOM} yields
\begin{align*}
\| \bfx^{k} - \tbfx^{k}\|_{2} \le \| \bfx^{k-1} - \tbfx^{k-1} \|_{2} + \varDelta t \sum_{m=1}^{s} \abs{b_m} \|\bfw^k_m - \tbfw^k_m \|_2 + \norm{\tbfr_{\bfx}^{k}(\tbfx^{k}) }_2.
\end{align*}
By inserting \cref{Eq:upper_bound} it follows
\begin{align*}
\| \bfx^{k} - \tbfx^{k}\|_{2} &\le \left( 1+ \kappa \varDelta t \sum_{m=1}^{s}|b_m|\sum_{i=1}^{s}[\bm{D}^{-1}]_{mi} \right) \| \bfx^{k-1} - \tbfx^{k-1} \|_{2} \\
&+\varDelta t \sum_{m=1}^{s}|b_m|\sum_{i=1}^{s}[\bm{D}^{-1}]_{mi}  {\|\tbfr^k_i(\tbfw_1^{k}, \ldots, \tbfw_s^{k}) \|_{2}} +\norm{\tbfr_{\bfx}^{k}(\tbfx^{k}) }_2.
\end{align*}
Finally, an induction argument concludes the proof. 
\end{proof}

Note, that the Lipschitz condition on $X_{\cH}$ does not pose a severely limiting additional requirement, as it is a typical condition for the well-posedness of \cref{Eq:FOM_canonical}. 
 
\section{Nonlinear Symplectic Trial Manifold based on Deep Convolutional Autoencoders}
\label{Sec:SymAuto}

It is yet open how to construct a symplectic reconstruction function $d: \R^{2n} \rightarrow \R^{2N}$ and how to choose the reduced initial value $\bfx_{r,0}(\bfmu) \in \R^{2n}$ in \cref{Eq:ROM_canonical}. In the scope of this work, we investigate autoencoders for this purpose. In this section, we first introduce existing autoencoders and deep convolutional autoencoders (DCA). Furthermore, we introduce our novel \emph{weakly symplectic DCA}. Before we start describing the general idea of autoencoders, we briefly highlight that autoencoders are not the sole choice to define a reconstruction function.

\begin{myrem}[Alternatives to Autoencoders]\label{Rem:AlternativesToAutoencoders}
Instead of relying on an autoencoder to provide the reconstruction function $d$, one could consider using other approaches. One possibility would be the kernel principal component analysis (kPCA) \cite{Schoelkopf98}, which has the basic idea of embedding a nonlinear manifold in $\R^{N}$ into $\R^{\cN}$ with $\cN \gg N$. The corresponding mapping and dimension $\cN$ should be chosen such that the nonlinear manifold in $\R^{N}$ can be represented as a linear manifold in $\R^{\cN}$. Subsequently, a classical PCA is applied to identify the low-dimensional space. Additional methods to construct such a mapping are, e.g., the Gaussian process latent variable model (GPLVM) \cite{lawrence2003gaussian}, self-organizing maps \cite{kohonen1982self}, diffeomorphic dimensionality reduction \cite{NIPS2008_647bba34} and deep kernel networks \cite{wenzel_sdkn_turbulence}.
\end{myrem}

\subsection{Autoencoder}
Autoencoders \cites{Hinton96}[Chapter 14]{Goodfellow-et-al-2016} \ are one approach in the class of (nonlinear) dimensionality reduction methods that learn a low-dimensional representation of a given, finite dataset $\dataset \subset \R^N$, $\abs{\dataset} < \infty$.
An autoencoder $\cA$ is characterized by a pair of parametric mappings $\cA := \lb e(\cdot; \aeparam), d(\cdot; \aeparam) \rb$ which are called the encoder $e(\cdot; \aeparam): \R^N \rightarrow \R^n$ and the decoder $d(\cdot; \aeparam): \R^n \rightarrow \R^N$.
Typically both maps are artificial neural networks (ANNs) where $\aeparam \in \Theta \subseteq \R^{\naeparam}$ describes the network parameters such as the weights and biases.
The encoder--decoder pair is optimized to minimize the loss
\begin{align}\label{Eq:Loss_Data}
	\lossData(\aeparam) := \frac{1}{N\abs{\dataset}}\sum_{\bfx \in \dataset} \norm{\bfx - d( e(\bfx; \aeparam); \aeparam)}^2,
\end{align}
i.e., the data $\bfx \in \dataset$ should be invariant under compression with the encoder and subsequent decompression with the decoder. With this construction, autoencoders do not require labeled data which is known as unsupervised learning. The constant $N\abs{\dataset}$ normalizes the squared loss which is known as the mean squared loss (MSE). The training process identifies a (most often locally) optimal parameter $\saeparam \in \Theta$.

\subsection{Deep Convolutional Autoencoder}
DCAs \cite{LC2020A} choose the training data as $\dataset = \dataset(\cPtrain)$ with
\begin{align}\label{Eq:Dataset}
	\dataset(\cPtrain)
:= \{
	\bfx^k(\bfmu) - \bfx_0(\bfmu)
	\;\big|\;
	0 \leq k \leq K,\; \bfmu \in \cPtrain
\}
\end{align}
where $\bfx^k(\bfmu) \approx \bfx(t^k; \bfmu)$ is a numerical approximation for the $k-$th time step $t^k $, $0 \leq k \leq K$ and $\cP_\textnormal{train} \subset \cP$, $\abs{\cP_\textnormal{train}} < \infty$, is a finite training parameter set. Note that all initial values in \cref{Eq:Dataset} are shifted to $\bfzero \in \dataset$ by construction. Thus, the training of the DCA includes the condition $(d \circ e)(\bfzero) \approx \bfzero$.
Based on the optimized network parameters $\saeparam \in \Theta$, the reconstruction function is then defined by $d(\bfx_r) = d(\bfx_r; \saeparam)$ and the reduced initial value is set to $\bfx_{r,0}(\bfmu) \equiv \bfx_{r,0} = e(\bfzero; \saeparam)$, cf.\ \cite{LC2020A}.
The specific choice of the reduced initial value together with $(d \circ e)(\bfzero) \approx \bfzero$ leads to a reference state $\bfxref(\bfmu) = \bfx_0(\bfmu) - (d \circ e)(\bfzero) \approx \bfx_0(\bfmu)$.
Thus, the reconstruction function $d$ has to describe deviations from the initial value in \cref{Eq:reconstructed_solution} which is consistent with the choice of training data in \cref{Eq:Dataset}.
Note that the training of an autoencoder for model reduction purposes uses the same data as the snapshot-based basis generation techniques in classical model reduction ($d$ linear), see \cite{LC2020A}.

For high-dimensional data, $N \gg 1$, DCAs purely based on fully-connected ANNs (FNNs) might be too expensive to be trained since $\naeparam \sim N$ or even $\naeparam \sim N^2$.
Moreover, FNNs do in general not make use of spatial correlation in the data.
For such problems, the use of convolutional ANNs (CNNs) \cites{lecun1998gradient}[Chapter 9]{Goodfellow-et-al-2016} is advantageous. These networks are based on the so-called cross-correlation which applies a filter matrix similar to the discrete convolution.
The dimension of the filter matrix is independent from the dimension of the input which makes it suitable for high-dimensional inputs.
Moreover, this operation is invariant with respect to translations due to its convolutional structure.

\subsection{Weakly Symplectic Deep Convolutional Autoencoder}

In general the decoder $d$ of the autoencoder learned with the DCA method will not be symplectic as required in \cref{Eq:Canonical_sympl}. Thus, we extent the DCAs to weakly symplectic DCAs by adding this constraint additionally to our loss function
\begin{subequations}
\begin{align}
\label{Eq:Loss}
	\loss(\aeparam) :=&\; \alpha \lossData(\aeparam) + (1-\alpha) \lossSympl(\aeparam),\\
\label{Eq:Loss_Sympl}
\lossSympl(\aeparam)
:=&\; \frac{1}{(2n)^2\abs{\dataset}}
\sum_{\bfx \in \dataset} 
\norm{
	({ \left.\Jac{\cdot; \aeparam}\right|_{e(\bfx; \aeparam)} }\rT)
	\J_{2N}
	\left.\Jac{\cdot; \aeparam}\right|_{e(\bfx; \aeparam)} {-} \J_{2n}
}_\textnormal{F}^2,
\end{align}
\end{subequations}
where $0 < \alpha < 1$ is a hyperparameter, $\norm{\cdot}_\textnormal{F}$ is the Frobenius norm, $(2n)^2\abs{\dataset}$ is again a normalizing constant and $\left.\Jac{\cdot; \aeparam}\right|_{e(\bfx; \aeparam)}$ should denote the evaluation of the Jacobian $\Jac{\cdot; \aeparam}$ at $e(\bfx; \aeparam)$.
With this choice, network parameters $\aeparam$ are penalized for which $d(\cdot; \aeparam)$ violates the canonical symplecticity \cref{Eq:Canonical_sympl} at the encoded data points $e(\bfx; \aeparam)$ for $\bfx \in \dataset$.

Note that the symplecticity loss \cref{Eq:Loss_Sympl} is more expensive to compute than the data loss \cref{Eq:Loss_Data} as it requires to compute the Jacobian of $d$.
In our framework, we use the Jacobian--vector product in \texttt{pyTorch} \cite{NEURIPS2019_bdbca288} to compute the Jacobian.
For now, we ignore this additional cost as it is part of the training phase and it does not influence the computational cost to evaluate the encoder and its Jacobian after the training.
Future work might investigate how to reduce the cost in \cref{Eq:Loss_Sympl}, e.g., by introducing another stochastic approximation such as Hutchinson's Trace Estimator~\cite{Hutchinson1989}.

\section{Numerical Results}
\label{Sec:NumExp}
In this section, we perform numerical simulations on a linear wave equation; a problem with known slowly decaying Kolmogorov-$n$-widths, see e.g., \cite{greif2019decay, glas2020reduced}. First, we detail our problem and write it in Hamiltonian form. Then, we describe the autoencoder training and the construction of the reduced model. Subsequently, we compare our reduced model on the nonlinear symplectic trial manifold to classical (non-)symplectic model reduction and to non-symplectic model reduction on manifolds. We investigate the preservation of energy and symplecticity of the reduced model. 

\subsection{Model Data}
\label{Sec:exp_lin_wave}
We consider a parametrized one-dimensional linear wave equation with homogeneous Dirichlet boundary conditions
\begin{subequations} \label{eq:lin_wave}
\begin{align}
	\partial^2_{tt} u(t, \xi; \mu)
=&\; \mu^2 \partial^2_{\xi\xi} u(t, \xi; \mu), && \text{on } I \times \varOmega, \\
	u(0, \xi; \mu)
=&\; u_0(\xi; \mu), && \forall \, \xi \in \varOmega, \\
	\partial_t u(0, \xi; \mu)
=&\; -\mu \; \partial_\xi u_0(\xi; \mu), && \forall \, \xi \in \varOmega, \\
	u(t, \xi; \mu) =&\; 0, && \forall \,t \in I, \xi \in \{-1/2 , 1/2\},
\end{align}
\end{subequations}
with $\varOmega = (-1/2, 1/2)$ and $I = (0,1)$. The parameter $\mu\in \cP:=[5/12,5/6] \subset \R$ is chosen to be the wave speed. The initial value $u_0(\xi; \mu) := h\lb s(\xi;\mu) \rb$ is based on the spline function
\begin{align*}
	h(s(\xi;\mu))
:= \begin{cases}
1 - 3/2 \cdot s(\xi;\mu)^2 + 3/4 \cdot s(\xi;\mu)^3,& 0 \leq s(\xi;\mu) \leq 1,\\
(2-s(\xi;\mu))^3/4,& 1 < s(\xi;\mu) \leq 2,\\
0,& \text{otherwise},\\
\end{cases}
\end{align*}
with $s(\xi;\mu) := 4/\mu \cdot \abs{\xi + 1/2 - \mu/2}$ for which the unique solution is a traveling wave solution $u(t, \xi; \mu) = u_0(\xi - \mu t; \mu)$.
This model problem is similar to \cite{PM2016,AH2017} and very challenging as it describes the transport of a thin pulse.
We rewrite \cref{eq:lin_wave} in canonical form and change the variables to $q(t, \xi; \mu):=u(t, \xi; \mu)$ as well as $p(t, \xi; \mu):=\partial_t q(t, \xi; \mu) = \partial_t u(t, \xi; \mu)$ and arrive at the Hamiltonian formulation
\begin{align}\label{eq:ham_eq}
\partial_t q(t,\xi; \mu) = p(t,\xi;\mu), \quad  \partial_t p(t,\xi;\mu) = -  \mu^2 \partial_{\xi\xi} q(t, \xi;\mu),
\end{align}
with corresponding continuous Hamiltonian
\begin{align*}
\cH_{\textnormal{cont}}(q,p;\mu) := \frac{1}{2} \int_{\varOmega} \mu^2 (\partial_{\xi} q(t,\xi;\mu))^2 + p(t,\xi;\mu)^2 \textd\xi.
\end{align*}
We discretize the inner of $\varOmega$ into $N=2048$ equidistantly spaced points $\xi_i:= i\varDelta_\xi - 1/2$ for $i = 1, \ldots, N$, $\varDelta_\xi:= 1/(N+1)$ and $q_i(t;\mu) := q(t,\xi_i;\mu) $, $p_i(t;\mu) := p(t,\xi_i;\mu)$, $i = 1, \ldots N$. Discretizing \cref{eq:ham_eq} using a finite difference method, we arrive at the FOM
\begin{align*}
	\ddt \bfx(t;\mu) = \J_{2N} \nabla_{\bfx}\cH (\bfx(t;\mu);\mu), 
\end{align*}
with $\bfx(t;\mu) := \rT{(q_1(t;\mu), \ldots, q_N(t;\mu),p_1(t;\mu), \ldots, p_N(t;\mu))}\in \R^{2N}$ and the discrete Hamiltonian
\begin{align*}
	\cH(\bfx(t;\mu);\mu):= \frac{1}{2} \rT{\bfx(t;\mu)}  \begin{pmatrix} -\mu^2 \bm{D}_{\xi \xi} & \bm{0}_N \\ \bm{0}_N & \mathbb{I}_{N} \end{pmatrix} \bfx(t;\mu),
\end{align*}
where $ \bm{D}_{\xi \xi}$ is the central finite difference approximation for the derivative $\partial_{\xi\xi}$. 
For the discretization in time, we need a symplectic integrator in order to preserve the symplecticity of the Hamiltonian system. To this end, we use the implicit midpoint rule \cref{eq:impl_midpoint} with $K=4000$ steps, such that $\varDelta t:= \frac{1}{K}, t^{k}:= k\varDelta t, 0 \le k \le K$, and seek an approximation $\bfx^k(\mu)\approx \bfx(t^{k};\mu)$. The training set $\cP_{\textnormal{train}}$ is chosen as 8 equidistant points in $\cP$ which results in $\abs{\cP_{\textnormal{train}}} \cdot K = 32000$ snapshots.

\subsection{Autoencoder Training}
In the following we discuss all details on architecture and training of the autoencoders used in the experiment. The settings specific for the different autoencoders are listed in \cref{Table:Training_Settings}. We assume that the reader is familiar with the basics of ANNs. For a more detailed introduction, we refer, e.g., to \cite{LC2020A}.

For the training, the snapshot data is randomly split into two disjoint sets, (i) training data $\datasetTrain$ ($80\%$) and (ii) validation data $\datasetVal$ ($20\%$).
The network architectures are (weakly symplectic) DCAs as described in \cref{Sec:SymAuto}. For nine different reduced sizes $2n \in \{2,4,\dots,12,18,24,30 \}=: \reddims$, we investigate a weakly symplectic DCA $\cAEs{2n}$ architecture, a non-symplectic copy $\cAE{0}{2n}$ with $\alpha = 1$ in \cref{Eq:Loss}, and a completely different non-symplectic DCA architecture $\cAE{1}{2n}$. The hyperparameters for $\cAE{s,0}{2n}$ were decided in a hyperparameter optimization run for $250$ different hyperparameters (with $\alpha < 1$) by minimizing the loss \cref{Eq:Loss} on the validation set $\datasetVal$ whereas the DCAs $\cAE{1}{2n}$ were determined analogously in a separate run with $\alpha=1$ over $400$ candidate hyperparameters.
For all different reduced dimensions $2n$, the number of network parameters $\naeparam$ is between $93,000$ and $101,000$ for $\cAE{s,0}{2n}, \cAE{0}{2n}$ and between $248,000$ and $257,000$ for $\cAE{1}{2n}$.

Both, encoder and decoder, are composed of multiple layers that are evaluated sequentially (see \cref{Fig:Architecture}). The sizes of the intermediate results are indicated in the upper part. Sizes with an $\times$-symbol indicate that the result is a second-order tensor. For such results, we call the first dimension the \emph{channels} and the second dimension the \emph{length}.
The outermost layers ($\splitting$, $\flatten$, dotted boundary around layers) convert the state $[\bfq; \bfp] \in \R^{2N}$ to a second-order tensor $\rTsb{\bfq, \bfp} \in \R^{2\times N}$ (or $\R^{2\times N} \rightarrow \R^{2N}$, respectively) such that both physical quantities, the configuration variables $\bfq$ and the conjugated momenta $\bfp$, occupy one separate channel.
The next part is framed by a scaling operation and its inverse ($\scale$, $\scale^{-1}$, dashed boundary around layers). The scaling is chosen to scale each channel in the input data separately to the interval $[0,1]$.
This operation is typically applied directly to the data.
In our approach, this operation is adopted in the network in order to ensure that this operation is respected in the symplecticity loss \cref{Eq:Loss_Sympl}.
The other layers in \cref{Fig:Architecture} are, in order, a block of convolutional layers ($\convblock$), a flatten operation, a block of fully-connected layers ($\fullblock_e$) in the encoder, a block of fully-connected layers ($\fullblock_d$) in the decoder, a splitting operation and a block of transposed convolutions ($\convTblock$).
Each block consists of alternating (transposed) convolutional layers (or fully-connected layers, respectively) and nonlinear activation layers. The nonlinear activation function in all examples is the Exponential Linear Unit (ELU)
\begin{align*}
	&\sigma(x)
:= \begin{cases}
	x, &x \geq 0,\\
	\exp(x)-1,& x < 0.
\end{cases}
\end{align*}

\begin{figure}
	\includegraphics[height=\textwidth, angle=90]{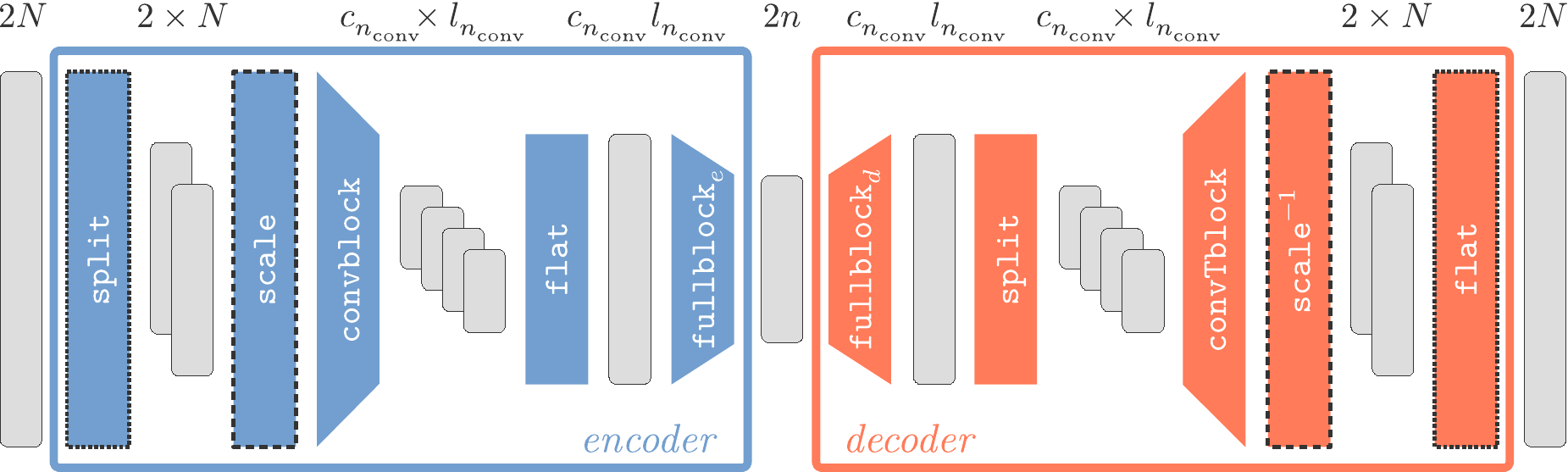}
	\caption{Schematic description of the network architecture of general DCAs composed of an encoder (blue) and a decoder part (orange). The sizes of the intermediate results (light gray) are indicated in the upper part. Sizes with an $\times$-symbol indicate that the result is a second-order tensor.}
	\label{Fig:Architecture}
\end{figure}

The hyperparameters for a (transposed) convolutional layer are the kernel size, the stride value $\stride_i$, the number of output channels $c_i$ and the padding. The length of the output $\lconv_{i+1} \in \N$ of a (transposed) convolutional layer depends on all those hyperparameters and the length of the input $\lconv_{i}$. The hyperparameters for the fully-connected layers are the number of neurons described by the lengths $\lfull_i$. We give an overview of the hyperparameters and the resulting lengths for each autoencoder in \cref{Table:Training_Settings}. Layer-specific hyperparameters are denoted as a vector $\bfv = [v_0, \dots, v_m]$. 

We use mini-batching in combination with the optimizer ADAM \cite{kingma2014adam} with standard parameters except for the learning rate, which is given together with the batch size in \cref{Table:Training_Settings} for each DCA. The batches are reshuffled in every epoch. The DCAs are trained for $1000$ epochs where one epoch is a full iteration over all training batches. The final parameter $\saeparam$ is chosen as the minimizer of the total loss \cref{Eq:Loss} over the validation data set $\datasetVal$ within those $1000$ epochs.
All biases are initialized with zeros. The weights are initialized with the method and corresponding distribution listed in \cref{Table:Training_Settings}.
The DCAs are implemented with the \texttt{pyTorch} \cite{NEURIPS2019_bdbca288} framework. Moreover, the floating point precision is set to \texttt{double}-precision, as the FOM and ROM work with this precision.

\begin{table}
\input{tables/manifold_settings.tex}
\centering
\caption{Autoencoder-specific settings for the weakly symplectic DCAs $\cAEs{2n}$ and the DCAs $\cAE{0}{2n}$, $\cAE{1}{2n}$. Note, that the DCAs $\cAEs{2n}$ and $\cAE{0}{2n}$ share all parameters except for the weight $\alpha$. We only report the layout of the encoder part as the decoder uses a mirrored layout.}
\label{Table:Training_Settings}
\end{table}

\subsection{Online Evaluation} The DCAs described in the previous section are compared for three fixed parameter instances $\muGenA := 0.51$, $\muGenB = 0.625$, $\mu_3=0.74$ that are all not contained in the training set $\cP_\textnormal{train}$, i.e., we investigate how well the DCAs generalize to unseen data. 
For the weakly symplectic DCAs, we use the SMG reduction technique introduced in \cref{Sec:SMG}. 
For the non-symplectic DCAs, we use the manifold Galerkin (MG) and the manifold LSPG (M-LSPG), both introduced in \cite{LC2020A}. 
Additionally, we display the results for a classical symplectic MOR method, the cotangent lift (CL) with symplectic Galerkin (SG) projection, which was observed to be the best reduction technique in the experiments for the linear wave equation in \cite{PM2016}.
Moreover, we present results for the proper orthogonal decomposition (POD) w.r.t.\ Galerkin (G) and least-squares Petrov-Galerkin (LSPG) projection, which are classical (non-symplectic) MOR techniques.
The different reduction techniques are summarized in \cref{Table:Reductions}.
\begin{table}
\begin{tabular}{l||c|c|c|c}
	\textbf{reduction method}
& \textbf{$d$ lin.}
& \textbf{sympl.}
& \textbf{approx.}
& \textbf{ref.} \\

\hline
	Symplectic Manifold Galerkin (SMG)
& no
& yes
& $\cAEs{2n}$
&  \Cref{Sec:SMG} \\
\hline
	Manifold Galerkin (MG)
& no
& no
&  \multirow{2}*{$\cAE{0}{2n}, \cAE{1}{2n}$}
&  \cite{LC2020A} \\
\cline{0-2} \cline{5-5}
	Manifold LSPG (M-LSPG)
& no
& no
&
& \cite{LC2020A} \\
\hline
	Symplectic Galerkin (SG)
& yes
& yes
& $\dVCL{2n}$ 
& \cite{PM2016}\\
\hline 
	POD--Galerkin (G)
& yes
& no
& \multirow{2}*{$\dVPOD{2n}$} 
&  \cite{kunisch2001galerkin} \\
\cline{0-2} \cline{5-5}
	least-squares Petrov-Galerkin (LSPG)
& yes
& no
& 
& \cite{carlberg2011efficient} \\
\end{tabular}
\centering
\caption{Summary of all investigated reduction techniques.} 
\label{Table:Reductions}
\end{table}

\cref{Fig:proj_rom} shows the projection error
\begin{align}
	\eproj(\mu)
&:= \sqrt{\frac{
\sum_{k=0}^{K} \norm{\bfxref(\mu) + (d \circ e)(\bfx^k(\mu) - \bfxref(\mu)) - \bfx^k(\mu)}^2
}{
\sum_{\bfx \in \dataset(\{ \mu \})} \norm{\bfx}^2
}}\; \label{eq:errs_proj}	
\end{align}
and the reduction error
\begin{align}
	\ered(\mu)
&:= \sqrt{\frac{
\sum_{k=0}^{K} \norm{\bfxref(\mu) + d(\bfx_r^k(\mu)) - \bfx^k(\mu)}^2
}{
\sum_{\bfx \in \dataset(\{ \mu \})} \norm{\bfx}^2
}} \label{eq:errs_red}
\end{align}
where $\bfx^k(\mu)$ is the solution of the FOM and $\bfx_r^k(\mu)$ is the corresponding reduced-order solution at time $t^k$.
The errors are displayed for the DCAs from \cref{Table:Training_Settings} and the classical approximations $\dVPOD{2n}$, $\dVCL{2n}$ each with its corresponding reduction technique(s) from \cref{Table:Reductions} for different reduced dimensions $2n \in \reddims$.
The projection error is a measure how well a solution can be approximated by the respective DCA whereas the reduction error shows the error of the corresponding ROM. The errors are considered separately for Galerkin-based (\cref{Fig:proj_rom}, left) and LSPG-based (\cref{Fig:proj_rom}, right) projection techniques. Errors are not depicted, (a) if the error is above $400\%$ or (b) if the corresponding ROM simulation run did not reach the absolute tolerance of $\eNum{1}{-8}$ within the maximum number of $15$ quasi-Newton iterations for some time step $t^k$, $0 \leq k < K$.

\begin{figure}[ht!]
\begin{center}
	\includegraphics{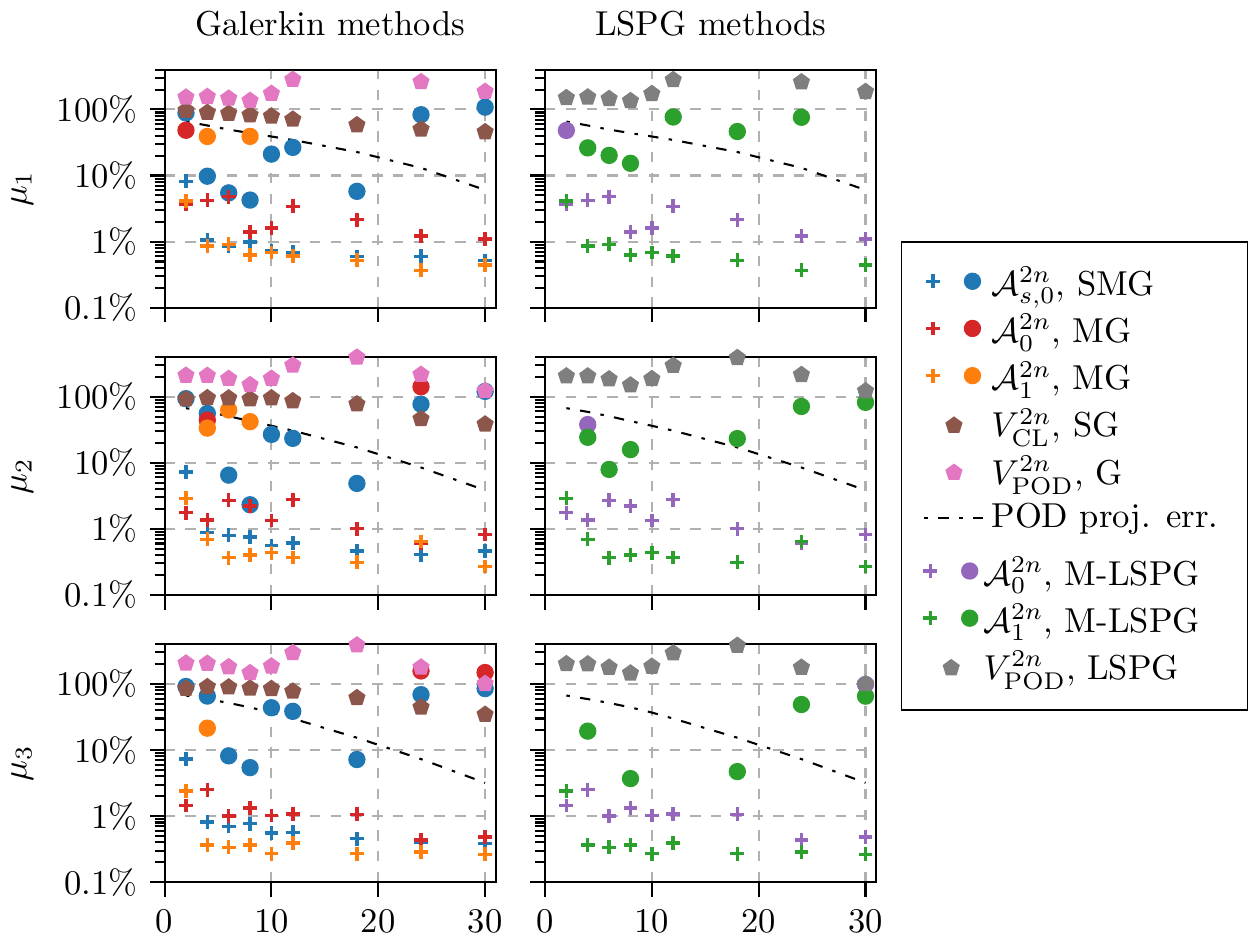}
	\caption{Reduction error $\ered(\mu)$ \cref{eq:errs_red} (circles and pentagon-symbols) and projection error $\eproj(\mu)$ \cref{eq:errs_proj} (plus-symbols and dash-dotted line) for three generalization parameters $\muGenA=0.51$ (top), $\muGenB=0.625$ (middle), $\muGenC=0.74$ (bottom) for different DCAs from \cref{Table:Training_Settings}, and reduction techniques from\cref{Table:Reductions}. We consider Galerkin (left) and LSPG methods (right) separately.}
\label{Fig:proj_rom}
\end{center}
\end{figure}
First, we comment on the classical MOR methods, i.e., CL with SG and POD with G and LSPG.
As expected from theory, the reduction error $\ered(\mu)$ of all linear, Galerkin-based methods (depicted by pentagon-symbols) is always bounded from below by the projection error $\eproj(\mu)$ of the POD basis (black dash-dotted line).
Moreover, all classical reduction methods show poor results in the reduction error, which is around or above $100\%$ for most $2n \in \reddims$. Overall, no reduction error is achieved below $34\%$ with classical methods for $2n \in \reddims$.
Therefore, in this setting classical MOR methods are not able to beneficially reduce the model for the considered reduced dimensions.

In contrast to that, all DCAs $\cAEs{2n}$, $\cAE{0}{2n}$, $\cAE{1}{2n}$ attain lower projection errors $\eproj(\mu)$ (plus-symbols) than the projection error $\eproj(\mu)$ for the POD basis for all configurations $\mu_i$, $i = 1,2,3$, $2n \in \reddims$, and a lower reduction error $\ered(\mu)$ (circles) for at least one configuration. The reduction errors provide results below and above the POD projection error, which aligns with previous results of performing model reduction on manifolds using autoencoders, see \cite{LC2020A}. For the Galerkin methods (left) with $4\le 2n < 24$, $\cAE{s,0}{2n}$ with SMG is always the best technique. Particularly for $2n \in \{6,8,18\}$ good reduction errors (below $10\%$) are obtained for all $\mu_i$, $i=1,2,3$. For the LSPG methods (right) only $\cAE{1}{2n}$, hyperoptimized for the non-symplectic setting, yields good approximations. However, even with $\cAE{1}{2n}$, no configuration yields errors below $10\%$ consistently for all considered parameters for all $\mu_i$, $i=1,2,3$. In that sense, the reduced models yield less reliable results than $\cAEs{2n}$ with SMG. Overall, the best reduction error is given by $\cAE{s,0}{2n}$ with SMG for $\muGenB$ with $2n=8$ with $2.3 \%$.

In comparison to the results from \cite{LC2020A} for the Burgers' equation, it is, however, surprising that the discrepancy between the reduction error and the projection error is very high for M-LSPG. This indicates that the lack of structure might impact M-LSPG which leads to higher reduction errors.\footnote{The MG and M-LSPG in our implementation have been validated with the Burgers' example from \cite{LC2020A} which leads to this conclusion.}

\begin{figure}[t!]
\begin{center}
	\begin{center}
	\includegraphics{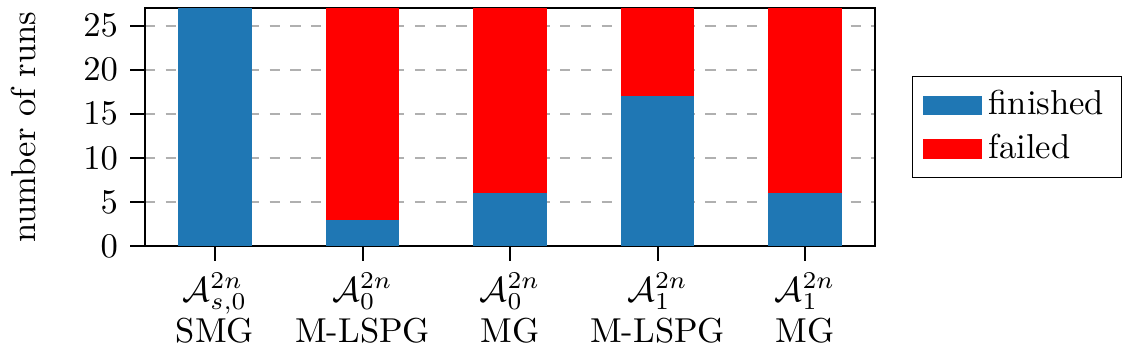}
	\end{center}
	\caption{Number of successful (blue) compared to failed (red) reduced simulation runs for $\cAE{s,0}{2n}, \cAE{0}{2n}, \cAE{1}{2n}$ with respective reduction technique, for $\mu_i$, $i=1, 2,3$ and $2n \in \reddims$.}
\label{Fig:bar_plot}
\end{center}
\end{figure}

For non-symplectic DCAs, we observe that non-converging reduced simulations occur, whereas all runs converge for the weakly symplectic DCAs, see \cref{Fig:bar_plot}. 
Out of $27$ simulations (9 reduced dimensions, 3 parameters) for $\cAE{1}{2n}$ , $6/27$ runs converged for MG and $17/27$ converged for M-LSPG. For $\cAE{0}{2n}$, these numbers are even lower with $6/27$ converged simulations for MG and $3/27$ converged simulations for M-LSPG. This means that, even in the best case, more than one third of the runs did not converge for non-symplectic DCAs. Thus we emphasize the importance of structure-preservation in model reduction. 

In \cref{fig:symplecticity}, we depict the error in the symplecticity of $\cAEs{2n}$ with
\begin{align} \label{eq:symplec_measure}
\esymp(t^k; \mu) := \frac{1}{(2n)^2}\| (\JacC_{\bfx_{r}}d(\bfx_r^k(\mu)))^{\textnormal{\textsf{T}}} \J_{2N} \JacC_{\bfx_{r}}d(\bfx_r^k(\mu)) - \J_{2n} \|_\textnormal{F}^{2},
\end{align}
where $\bfx_r^k(\mu)$ denotes the solution of the SMG-ROM.
It shows that the symplecticity is worst for $2n=2$ which we attribute to a too low reduced dimension in order to capture the essential properties of the solution manifold.
The error in the symplecticity behaves best for $2n \in \{ 6,8 \}$ and tends to increase for higher reduced dimensions.
The degeneration in the symplecticity might be an explanation for the increase in the reduction error (\cref{Fig:proj_rom} on the left).
We, thus, suspect that DCAs with improved symplecticity could yield more stable reduction errors.
The following remark briefly comments on a possibility to enhance symplecticity in DCAs.

\begin{myrem}\label{Rem:SympNet}
Another possibility to design an autoencoder which defines a symplectic reconstruction function $d$ would be a combination of the DCA-concept with the idea of SympNets introduced in \cite{Jin2020SympNetsIS} (which has been proposed in a similar fashion earlier, e.g., in \cite{Deco1995a}).
With such an approach, the symplecticity is encoded in the architecture of the autoencoder instead of implementing it with an additional loss term that penalizes non-symplectic DCAs (as in \cref{Eq:Loss_Sympl}).
In the scope of the present work, we opted for the latter approach to be able to compare the SMG-ROM with the MG-ROM on comparable architectures. 
\end{myrem}

\begin{figure}[t!]
\begin{center}
	\subfigure[\label{fig:symplecticity} Error in symplecticity $\esymp(t^k; \muGenA)$ \cref{eq:symplec_measure}.]{%
	\includegraphics{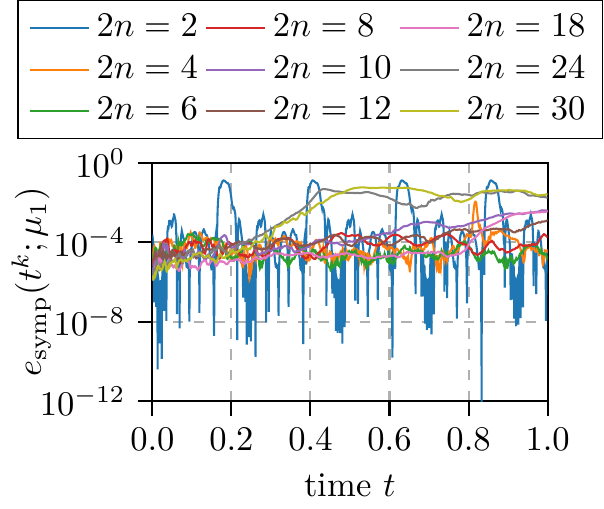}%
	}
	\subfigure[\label{fig:Hamiltonian} Error in Hamiltonian $\varDelta\mathcal{H}(t^k; \muGenA)$ \cref{eq:err_in_hamiltonian}.]{%
	\includegraphics{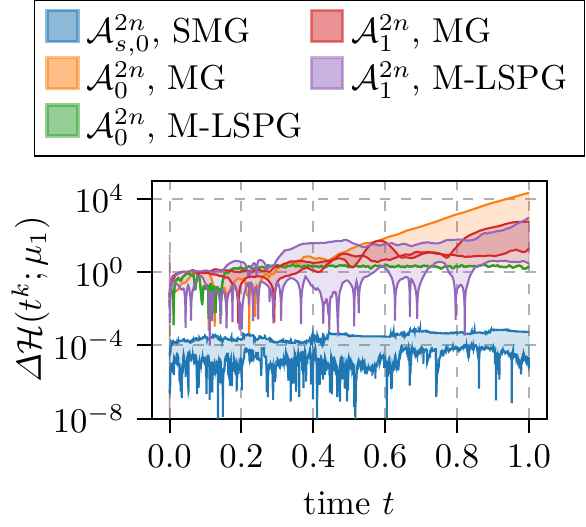}%
	}
	\caption{Error in the symplecticity $\esymp(t^k; \muGenA)$ \cref{eq:symplec_measure} of $\cAEs{2n}$ for $2n \in \reddims$, $t \in {[0,1]}$ (left) and variation for the error in the Hamiltonian $\varDelta\mathcal{H}(t^k; \muGenA)$ \cref{eq:err_in_hamiltonian} of $\cAE{s,0}{2n}, \cAE{0}{2n}, \cAE{1}{2n}$, for $2n \in \reddims$, $t \in {[0,1]}$ (right). For $\cAE{0}{2n}$ with M-LSPG (green) only one simulation run converged such that there is a line instead of two lines enclosing a shaded area. For $t\ge0.3$, the lower yellow and green line coincide.}
\label{Fig:Hamiltonian_symplecticity}
\end{center}
\end{figure}

In \cref{fig:Hamiltonian}, we plot the evolution of the error in the Hamiltonian $\varDelta \cH(t;\muGenA)$ \cref{eq:err_in_hamiltonian} for the DCAs $\cAEs{2n}, \cAE{0}{2n}$ and $\cAE{1}{2n}$ (with SMG, MG, M-LSPG, whenever applicable). It shows the minimum and maximum of $\varDelta \cH(t;\muGenA)$ over different reduced dimensions $2n \in \reddims$ for each time step $t^k$, $0 \leq k \leq K$, as solid lines and the area between both is shaded in the same color.
Non-converging runs (red in \cref{Fig:bar_plot}) are excluded, which is in favor of the non-symplectic DCAs $\cAE{0}{2n}, \cAE{1}{2n}$.
We observe that the error in the Hamiltonian for small times $t\approx 0$ is very low in all cases. This is due to the specific choice of the reference value $\bfxref(\bfmu)$ in \cref{rem:exact_reproduction} which makes the DCAs to match the initial value, and thus the Hamiltonian at $t=0$, exactly. All runs of the weakly symplectic DCAs $\cAEs{2n}$ are able to retain a low error in the Hamiltonian ($\varDelta\mathcal{H}(t^k; \muGenA) < \eNum{5.5}{-4}$).
For the non-symplectic DCAs $\cAE{0}{2n}, \cAE{1}{2n}$, the error in the Hamiltonian is much higher.
In the best case, the minimum error in the Hamiltonian at $t=1$ is about $2$.
In that sense, the additional constraint of weak symplecticity in \cref{Eq:Loss_Sympl} in combination with the SMG projection shows its impact in this experiment with an improved preservation of the Hamiltonian compared to the non-symplectic DCAs.

\section{Conclusion}
\label{Sec:conclusions}

In this paper, we proposed a method for symplectic model reduction on manifolds. In particular, we developed the SMG projection technique, which projects a Hamiltonian system onto a nonlinear symplectic trial manifold, resulting in a low-dimensional Hamiltonian system. Additionally, we presented a rigorous error estimator and proved that the reduced model preserves energy and, under mild assumptions, also stability in the sense of Lyapunov. By using a novel weakly symplectic DCA for the construction of the reduced model, we were able to promote symplecticity in the reduced model. Note, that for the generation of the reduced model, the same snapshots were used as for the classical symplectic model reduction. Numerical simulations for a challenging case of a moving pulse governed by a wave equation showed that the SMG approach is superior over classical (symplectic) model reduction in the sense that a very low dimension of the manifold is enough to arrive at a low reduction error. Moreover, numerical experiments indicated that model reduction on manifolds, which does not promote symplecticity, might not yield a good reduced solution - similar to the comparison of symplectic model reduction to classical model reduction for Hamiltonian systems, see \cite{PM2016,AH2017}. We note that our method inherits similar drawbacks as the method in \cite{LC2020A}, i.e., requiring more hyperparameters than classical model reduction due to the autoencoder architecture and no full offline--online separation. To address the latter, future work will focus on including a hyper-reduction framework motivated by \cite{KCWZ2020B}. Additional future work will also include designing a novel symplectic autoencoder architecture based on SympNets \cite{Jin2020SympNetsIS}, such that symplecticity can directly be included in the architecture of the autoencoder, see \cref{Rem:SympNet}.\\[1em]


\acknowledgements{S.~Glas acknowledges support from the Simons Foundation in the Collaboration on Hidden Symmetries and Fusion Energy. P.~Buchfink and B.~Haasdonk are funded by Deutsche Forschungsgemeinschaft (DFG, German Research Foundation) under Germany's Excellence Strategy - EXC 2075 – 390740016. Both acknowledge support by the Stuttgart Center for Simulation Science (SimTech).}

\printbibliography

\end{document}

%% file: tables/manifold_settings.tex

\begin{tabular}{l||c|c|c}
	\textbf{setting} & $\quad\;\;\bcAEs{2n}\quad\;\;$ & $\bcAE{0}{2n}$ & $\bcAE{1}{2n}$\\
\hline
	weight $\alpha$ in \cref{Eq:Loss} & 0.9 & \multicolumn{2}{c}{1.}\\
\hline
	num. conv. layer $\nconv$
& \multicolumn{2}{c|}{6}
& 5 \\
\hline
	conv. channels $\csconv$
& \multicolumn{2}{c|}{$[2,2,4,8,16,32,64]$}
& $[2,4,8,16,32,64]$\\
\hline
	lengths, conv. layer $\lsconv\!\!$
& \multicolumn{2}{c|}{$\!\![2048,512,256,128,64,16,2]\!\!$}
&$\!\![2048,1024,512,256,128,4]\!\!\!$\\
\hline
	stride $\strides$
& \multicolumn{2}{c|}{$[4,2,2,2,4,8]$}
& $[2,2,2,2,32]$\\
\hline
	num. full layer $\nfull$
& \multicolumn{2}{c|}{1}
& \multicolumn{1}{c}{2} \\
\hline
	lengths, full layer $\lsfull$
& \multicolumn{2}{c|}{[$128, 2n$]}
& \multicolumn{1}{c}{[$256, 132, 2n$]} \\
\hline
	learning rate
& \multicolumn{2}{c|}{\eNum{4.43}{-4}}
& \eNum{1.05}{-4}\\
\hline
	batch size
& \multicolumn{2}{c|}{15}
& \multicolumn{1}{c}{25}\\
\hline
	initialization
& \multicolumn{2}{c|}{Kaiming normal \cite{he2015delving}}
& \multicolumn{1}{c}{Xavier uniform \cite{pmlr-v9-glorot10a}}\\
\end{tabular}